\documentclass{amsart}
\usepackage[style=alphabetic]{biblatex} 
\usepackage{tikz}
\usepackage{amssymb}
\usepackage{amsthm}
\usepackage[all]{xy}
\usepackage{microtype}
\usepackage{xcolor}
\usepackage{hyperref}
\usepackage[nameinlink]{cleveref}

\hypersetup{
 colorlinks,
 linkcolor={teal},
 citecolor={teal},
 urlcolor={teal}
}

\calclayout

\DeclareFieldFormat
  [article,book,inbook,incollection,inproceedings,patent,thesis,unpublished]
  {title}{\emph{#1\isdot}}

\theoremstyle{plain}
	\newtheorem{theorem}{Theorem}
	\newtheorem{lemma}[theorem]{Lemma}

	\newtheorem{proposition}[theorem]{Proposition}
	
        \newtheorem{problem}[theorem]{Problem}
        \newtheorem{propdef}[theorem]{Proposition-Definition}

\theoremstyle{definition} 
	
	\newtheorem{definition}[theorem]{Definition}

\everymath{\displaystyle}

\begin{document}
\title{Asymptotic Analysis of Infinite Decompositions of a Unit Fraction into Unit Fractions}
\author[Y. Kamio]{Yuhi Kamio}
\address{College of Arts and Sciences, University of Tokyo, 3-8-1 Komaba, Meguro-ku, Tokyo 153-8914, Japan}
\email{emirp13@g.ecc.u-tokyo.ac.jp}

\date{\today}

\begin{abstract}
Paul Erdős posed a problem on the asymptotic estimation of decomposing 1 into a sum of infinitely many unit fractions in \cite{Erd80}. We point out that this problem can be solved in the same way as the finite case, as shown in \cite{Sou05}. 
\end{abstract}

\maketitle
\setcounter{tocdepth}{1}

\enlargethispage*{20pt}
\thispagestyle{empty}

\begin{section}{introduction}

In Section 4 of \cite{Erd80}, Paul Erdős posed many problems about decomposing numbers into sums of unit fractions. One of these problems is the following:

\begin{problem}[\cite{Erd80}, p.41]\label{main_problem}\footnote{We correct the erratum as in \cite{erd_pr}. In the original literature, $u_{i+1}$ was defined as $u_i(u_i+1)$ instead of $u_i(u_i+1)+1$.}
Let $u_i$ be the Sylvester sequence, i.e., $u_0=1$ and $u_{i+1}=u_i(u_i+1)+1$. Then, we have $\displaystyle{\sum_{i=1}^{\infty} \frac{1}{u_i}=1}$.
If $a_1 \leq a_2 \leq a_3 \leq \cdots$ is another sequence with $\displaystyle{\sum_{i=1}^{\infty} \frac{1}{a_i}=1}$, is it true that
    \begin{equation*}
        \liminf_{i\to\infty} a_i^{2^{-i}} < \lim_{i\to \infty} u_i^{2^{-i}} = 1.2640...\text{?}
    \end{equation*}
\end{problem}
This problem remains unsolved to this day. 

On the other hand, it is natural to consider the "finite version" of this problem, i.e.:
\begin{problem}
    Assume that $a_1 \leq a_2 \leq a_3 \leq \cdots \leq a_n$ satisfies $\displaystyle{\sum_{i=1}^n \frac{1}{a_i}=1}$. Is it true that
        \begin{equation*}
            a_n \leq u_n - 1
        \end{equation*}
    and does equality hold when $a_i = u_i$ for all $i < n$?
\end{problem}
This problem has been solved in various ways (\cite{Tak21}, \cite{Cur22}, \cite{Sou05}). In this paper, we solve \cref{main_problem} and its generalization (\cref{main}) based on the arguments in \cite{Sou05}.
\end{section}

\begin{section}{proof of main theorem}

\begin{definition}[Generalized Sylvester sequence]
    Let $n$ be a positive integer. We define the sequence $s_i(n)$ by
    \begin{align*}
        s_1(n) &= n+1, \\
        s_{i+1}(n) &= s_i(n)^2 - s_i(n) + 1.
    \end{align*}
\end{definition}

The sequence $s_i(1)$ is the ordinary Sylvester sequence.

We enumerate some properties of the generalized Sylvester sequence.

\begin{propdef}\label{prop_cn}
    Let $n$ be a positive integer. Then,
    \begin{enumerate}
        \item $\lim_{i\to\infty} s_i(n)^{2^{-i}}$ exists. We denote this limit by $c_n$.
        \item $\sqrt{n} < c_n < \sqrt{n+1}$. In particular, if $n < l$ for a positive integer $l$, then $c_n < c_l$.
    \end{enumerate}
\end{propdef}
\begin{proof}
    (1): As $s_{i+1}(n)=s_i(n)^2-s_i(n)+1 < s_i(n)^2$, $
    \{s_i(n)^{-2^i}\}_i$ is strictly decreasing sequence. Therefore, $\lim_{i\to \infty} s_i(n)^{-2^i}$ exists and $c_n<s_1(n)^{2^{-1}}=\sqrt{n+1}$. \\
    (2): As $s_{i+1}(n)-1=s_i(n)^2-s_i(n) > (s_i(n)-1)^2$, $
    \{(s_i(n)-1)^{-2^i}\}_i$ is strictly increasing sequence. Therefore, $c_n>(s_1(n)-1)^{2^{-1}}=\sqrt{n}$, and the first half of proposition follows.
    The last half follows easily from the first. 
\end{proof}
\begin{proposition}\label{prop_syl}
Let $n, j$ be positive integers. Then, the following statements hold:
\begin{enumerate}
    \item $\sum_{i=1}^{j-1} \frac{1}{s_i(n)}+\frac{1}{s_j(n)-1}=\frac{1}{n}.$
    \item $s_{j}(n)-1=n\prod_{i=1}^{j-1} s_i(n).$
    \item Let $i$ be a positive integer. Then, $s_{i+j-1}(n)=s_i(s_j(n)-1).$
    \item $c_n^{2^{j-1}}=c_{s_j(n)-1}.$
\end{enumerate}
\end{proposition}

\begin{proof}
    \begin{enumerate}
        \item This follows from $\frac{1}{s_j(n)-1}-\frac{1}{s_j(n)}=\frac{1}{s_j(n)^2-s_j(n)}=\frac{1}{s_{j+1}(n)-1}$ and induction.
        \item This follows from $s_{j+1}(n)-1=(s_j(n)-1)s_j(n)$ and induction.
        \item Use induction on $i$.
        \item
            \begin{align*}
                c_n^{2^j} &= \lim_{i\to\infty} s_{i}(n)^{2^{j-1-i}} 
                = \lim_{i\to\infty} s_{i+j-1}(n)^{2^{-i}} 
                = \lim_{i\to\infty} s_{i}(s_j(n)-1)^{2^{-i}}
                = c_{s_j(n)-1}.
            \end{align*}
    \end{enumerate}
\end{proof}
The next lemma is essentially the same as the argument in \cite{Sou05}. 

\begin{lemma}\label{pac_sou}  
    Let $u$ be a positive real number, and let $t$ be a positive integer. Let $u_1 \geq u_2 \geq \cdots > 0$ be a sequence of real numbers that satisfies the following equation for $n \geq t$:  
    \[
    \sum_{i=1}^n u_i + u \prod_{i=1}^n u_i = u  
    \]  
    Let $v_1 \geq v_2 \geq \cdots > 0$ be a sequence of real numbers that satisfies the following inequality for $n \geq t$:  
    \[
    \sum_{i=1}^n v_i + u \prod_{i=1}^n v_i \leq u  
    \]  
    Also, assume that $v_n \leq u_n$ for $n < t$. Then, for all $n$,  
    \[
    \sum_{i=1}^n v_i \leq \sum_{i=1}^n u_i.  
    \]  
\end{lemma}

\begin{proof}
    Using induction on $n$. If $n < t$, this inequality is trivial. Assume that this inequality holds for all numbers less than $n$, and that $n \geq t$. If $v_n \leq u_n$, then this follows by the induction hypothesis for $n - 1$. If $\prod_{i=1}^n v_i \geq \prod_{i=1}^n u_i$, then 
    \[
        \sum_{i=1}^n v_i \leq u - u \prod_{i=1}^n v_i \leq u - u \prod_{i=1}^n u_i = \sum_{i=1}^n u_i.
    \]
    So we can assume that $v_n > u_n$ and $\prod_{i=1}^n v_i < \prod_{i=1}^n u_i$. Let $m$ be the maximum integer that satisfies $\prod_{i=m}^n v_i < \prod_{i=m}^n u_i$. By the maximality of $m$, for all $l \geq m$, we have $\prod_{i=m}^l v_i < \prod_{i=m}^l u_i$. By the 
    next lemma, $\sum_{i=m}^n v_i \leq \sum_{i=m}^n u_i$. Combining the induction hypothesis for $m - 1$, we get the desired inequality.
\end{proof}
\begin{lemma}
    Let $x_1 \geq x_2 \geq \cdots \geq x_n > 0$ and $y_1 \geq y_2 \geq \cdots \geq y_n > 0$ be real numbers. Assume that $\prod_{i=1}^j x_i \geq \prod_{i=1}^j y_i$ for all $j\leq n$. Then, 
    \[
        \sum_{i=1}^n x_i \geq \sum_{i=1}^n y_i.
    \]
\end{lemma}
\begin{proof}
    This lemma is identical to the proposition in \cite{Sou05}.
\end{proof}

This is the main theorem of this paper. 
\begin{theorem}\label{main}
    Let $n$ be a positive integer, and let $a_1 \leq a_2 \leq \cdots$ be a sequence of positive integers. Assume that $a_i \neq s_i(n)$ for some $i$, and that  
    \[
    \frac{1}{n} = \sum_{i=1}^\infty \frac{1}{a_i}.
    \]  
    Then,  
    \begin{equation*}
        \liminf_{i\to\infty} a_i^{2^{-i}} < c_n.
    \end{equation*}
\end{theorem}

\begin{proof}
    It is enough to consider the case where $a_1 \neq n+1$ by shifting the sequence. Indeed, assume that we have proven this theorem when $a_1 \neq n+1$.  
    Let $k$ be the smallest integer such that $a_k \neq s_k(n)$. Define $a'_i = a_{i+k-1}$ and $n' = s_k(n) - 1$. Then, by \cref{prop_syl} (1),  
    \[
        \sum_{i=1}^{\infty} \frac{1}{a'_i} =
        \frac{1}{n} - \sum_{i=1}^{k-1} \frac{1}{a_i} =
        \frac{1}{n} - \sum_{i=1}^{k-1} \frac{1}{s_i(n)} = \frac{1}{n'}.
    \]
    Also, by definition, $a'_1 = a_k \neq s_k(n) = n' + 1$. Therefore, by the hypothesis,  
    \[
        \liminf_{i\to \infty} {a'_i}^{2^{-i}} < c_{n'}.
    \]
    Raising both sides to the power of $2^{-(k-1)}$, we obtain  
    \[
        \liminf_{i\to\infty} a_i^{2^{-i}} < c_{n'}^{2^{-(k-1)}} = c_n.
    \]
    (The last equality follows from \cref{prop_syl} (4).)
    
    First, we consider the case $n \neq 1$. We define an integer $l$ as  
    $
        \frac{n(n+1)^2(n+2)}{2}
    $
    and define the sequence $u_i$ as follows:  
    \begin{align*}
        u_i =
        \begin{cases}
            \dfrac{1}{n+2} \vphantom{\dfrac{1}{s_{i-2}(l)}} & i=1, \\
            \dfrac{2}{(n+1)^2} \vphantom{\dfrac{1}{s_{i-2}(l)}} & i=2, \\
            \dfrac{1}{s_{i-2}(l)} & i \geq 3.
        \end{cases}
    \end{align*}

    Let $v_i = \frac{1}{a_i}$ for $i \in \mathbb{Z}_{>0}$, and set $u = \frac{1}{n}$, $t = 2$. We will verify that $u, u_i, v_i, t$ satisfy the conditions of \cref{pac_sou}. 
    Note that direct calculation shows  
    \[
        u = u_1 + u_2 + \frac{1}{l}, \quad uu_1u_2 = \frac{1}{l}.
    \]

    \begin{itemize}
        \item First equation:  
            Assume $m \geq t = 2$. Then,  
            \begin{align*}
                u - \sum_{i=1}^m u_i - u\prod_{i=1}^m u_i 
                &= \frac{1}{l} - \sum_{i=3}^m u_i - \frac{1}{l} \prod_{i=3}^m u_i \\
                &= \frac{1}{l} - \sum_{i=1}^{m-2} \frac{1}{s_{i}(l)} - \prod_{i=0}^{m-2} \frac{1}{s_{i}(l)} = 0.
            \end{align*}
            Here, we use (1) and (2) of \cref{prop_syl}. 
        \item Second inequality:  
            Assume $m \geq t = 2$ and let  
            \[
                q := \frac{1}{n} - \sum_{i=1}^m \frac{1}{a_i} =     \sum_{i=m+1}^{\infty} \frac{1}{a_i} > 0.
            \]  
            Since $q\bigl(n{\prod_{i=1}^m a_i}\bigr)$ is a positive integer, we have  
            \[
                u - \sum_{i=1}^m v_i = q \geq (n\prod_{i=1}^m a_i)^{-1} = u\prod_{i=1}^m v_i.
            \]
            This shows the second inequality.  
    
        \item Third inequality:  
            Assume $m < t = 2$. Then $m = 1$, and since we assume $a_1 \neq s_1(n)$, we have $n+2 \geq a_1$. Therefore, we obtain $u_1 \leq v_1$.  
    \end{itemize}

    Therefore, by \cref{pac_sou}, we have $\sum_{i=1}^n v_i \leq \sum_{i=1}^n u_i$ for all $n \geq t$. Combining this with the fact that $\sum_{i=1}^{\infty} v_i = u = \sum_{i=1}^{\infty} u_i$, it follows that $v_i \geq u_i$ for infinitely many $i$.  
    Thus,  
    \[
        \liminf_{i\to\infty} a_i^{2^{-i}} = \limsup_{i\to\infty} v_i^{-2^{-i}} \leq \lim_{i\to\infty} u_i^{-2^{-i}}
        = \lim_{i\to\infty} (s_{i-2}(l))^{2^{-i}} = c_l^{\frac{1}{4}}.
    \]  
    By \cref{prop_syl} (4), we have  
    $
        c_n^4 = c_{s_3(n)-1} = c_{n(n+1)(n^2+n+1)}.
    $
    Therefore, by \cref{prop_cn} (2), it is enough to show that  
    \[
        n(n+1)(n^2+n+1) > l = n(n+1) \frac{n^2+3n+2}{2}.
    \]  
    This follows from $n \geq 2$.

    For the case $n=1$, we define $u_i$ as follows:
    \begin{align*}
        u_i=
        \begin{cases}
        \dfrac{1}{3} \vphantom{\dfrac{1}{s_{i-3}(30)}} & i=1,2,\\
        \dfrac{3}{10} \vphantom{\dfrac{1}{s_{i-3}(30)}} & i=3,\\
        \dfrac{1}{s_{i-3}(30)} & i \geq 4.
        \end{cases}
    \end{align*}
    We can prove the statement using the same argument as above. (We apply \cref{pac_sou} for $t=3$.)

\end{proof}

\end{section}


\printbibliography
\end{document}